\documentclass[12pt]{article}

\usepackage{amsmath}
\usepackage{amsthm}
\usepackage{amsxtra}
\usepackage{amsfonts,amssymb}
\newtheorem{Th}{Theorem}

\newtheorem{Lm}[Th]{Lemma}

\theoremstyle{definition}

\date{}
\title{Any nonsingular action of the full symmetric group is isomorphic  to an action with invariant measure}

\author{Nessonov ~N.~I. }

\begin{document}
\maketitle
\begin{abstract}
Let $\overline{\mathfrak{S}}_\infty$ denote the set of all bijections of natural numbers. Consider the action of $\overline{\mathfrak{S}}_\infty$  on a {\it measure space } $\left( X,\mathfrak{M},\mu \right)$, where $\mu$ is $\overline{\mathfrak{S}}_\infty$-{\it{quasi-invariant}} measure. We prove that there exists $\overline{\mathfrak{S}}_\infty$-invariant measure equivalent to $\mu$.
\end{abstract}
\subsection{Introduction}
Let $\mathbb{N}$ be the set of all natural numbers  and let $\overline{\mathfrak{S}}_\infty$ be the group of all bijections of $\mathbb{N}$. This group is called  {\it infinite full symmetric group}.
To the given element $s\in\overline{\mathfrak{S}}_\infty $ we put  ${\rm supp}\,s=\left\{ n\in \mathbb{N}: s(n)\neq n \right\}$. Element $s\in\overline{\mathfrak{S}}_\infty $ is called finite if $\# {\rm supp}\,s< \infty$.  The set of all finite elements form {\it infinite symmetric group}  $\mathfrak{S}_\infty$.

Let ${\rm Aut}\,\left( X,\mathfrak{M},\mu \right)$ be the set of all {\it nonsingular} automorphisms of the measure space $\left( X,\mathfrak{M},\mu \right)$. We would recall that automorphism $\left( X,\mu \right)\stackrel{T}{\mapsto} (X,\mu)$ is {\it nonsingular} if for each measurable $Y\in X$, $\mu(TY)=0$ if and only if $\mu(Y)=0$.   Throughout this paper we suppose that $\mathfrak{M}$ is {\it countable generated} $\sigma$-algebra of measurable subsets of $X$. A homomorphism $\alpha$ from a group $G$ into ${\rm Aut}\, \left( X,\mathfrak{M},\mu \right)$ is called an action of $G$  on $\left( X,\mathfrak{M},\mu \right)$. For convenience we consider  $\alpha$ as the right action of the group  $G$ on $X$: $X\ni x\stackrel{\alpha_g}{\mapsto} xg\in X $, $g\in G$. We suppose that
\begin{eqnarray*}
\mu\left( \left\{ x\in X: x(gh)\neq (xg)h \right\} \right)=0 \text{ for each fixed pair } g,h\in G \text{ and }
\end{eqnarray*}
 $Ag^{-1}\in \mathfrak{M}$ for all $A\in\mathfrak{M}$, $g\in G$.
 Introduce measure $\mu\circ g$ by
\begin{eqnarray*}
\mu\circ g(A)=\mu(Ag), A\in \mathfrak{M}.
\end{eqnarray*}

Suppose  that measures $\mu$  and $\mu\circ g$ are equivalent (i.e. mutually absolutely continuous) for every $g\in G$. In this case measure $\mu$ is called $G$-quasi-invariant. Considering  the whole  equivalence class of measures $\nu$, equivalent to $\mu$ (the measure class  $\mu$), it is also  the same to say that the action preserves the class as a whole, mapping any such measure to  another such. Let $\frac{{\rm d}\,\mu\circ g}{{\rm d}\,\mu}$ denote  the Radon-Nikodym density  of $\mu\circ g$ with respect to $\mu$. For convenience we put $\rho(g,x)=\sqrt{\frac{{\rm d}\,\mu\circ g}{{\rm d}\,\mu}}(x)$. Then
\begin{eqnarray}\label{Radon_Nikodym}
\int\limits_X (\rho(g,x))^2 f(xg)\,{\rm d}\,\mu=\int\limits_X  f(x)\,{\rm d}\,\mu \;\;\; \text{ for all } f\in L^1(X,\mu).
\end{eqnarray}
\begin{Th}\label{abelian_main_th}
Let the action of $\overline{\mathfrak{S}}_\infty$ on  $\left( X,\mathfrak{M},\mu \right)$ is measurable.
If measure $\mu$ is $\overline{\mathfrak{S}}_\infty$-quasi-invariant and $\sigma$-algebra $\mathfrak{M}$ is countably generated then there exists $\overline{\mathfrak{S}}_\infty$-invariant measure $\nu$ (finite or infinite) equivalent to $\mu$.
\end{Th}
\subsubsection{Outline of the proof of Theorem \ref{abelian_main_th}.}\label{outline_abelian}
Since the action $X\ni x\mapsto xg\in X$, $g\in\overline{\mathfrak{S}}_\infty$ preserves the measure class $\mu$, we can to define the Koopman representation of $\overline{\mathfrak{S}}_\infty$ associated to this action. It is given in the space $L^2(X,\mu)$ by the unitary operators
\begin{eqnarray*}
\left( \mathcal{K}(g)\eta \right)(x)=\rho(g,x)\eta(xg), \text{ where } \eta\in L^2(X,\mu).
\end{eqnarray*}
    From the separability of $\sigma$-algebra $\mathfrak{M}$ follows the separability of the unitary group of the space   $L^2(X,\mu)$
 in the strong operator topology. Therefore, homomorphism $\mathcal{K}$ induces the separable topology on $\overline{\mathfrak{S}}_\infty$. But, by Theorem 6.26 \cite{Kech_Rosendal}, $\overline{\mathfrak{S}}_\infty$ has exactly two separable group topologies. Namely,  trivial and the usual Polish  topology, which is   defined by fundamental system of neighborhoods $\mathfrak{S}(n,\infty)=\left\{ s\in\overline{\mathfrak{S}}_\infty: s(k)=k \text{ for } k=1,2,\ldots,n \right\}$ of unit. Therefore, the representation $\mathcal{K}$ is continuous. It follows  that there exist $n\in\mathbb{N}\cup 0$ and non-zero $\xi\in L^2(X,\mu)$ with the property
 \begin{eqnarray}\label{fixed}
  \mathcal{K}(g)\xi=\xi \text{ for all } g\in\mathfrak{S}(n,\infty).
 \end{eqnarray}
 Set $E=\left\{ x\in X: \xi(x)\neq0  \right\}$. Using (\ref{fixed}), we obtain
 \begin{eqnarray}\label{invariant_n}
 \mu(E\Delta(Eg))=0 \text{ for all } g\in\mathfrak{S}(n,\infty).
 \end{eqnarray}
 For $A\subset E$ we define measure $\nu$ by
 \begin{eqnarray*}
 \nu(A)=\int\limits_X \chi_A(x)\cdot|\xi(x)|^2{\rm d}\,\mu.
 \end{eqnarray*}
 It follows from (\ref{fixed}) and (\ref{invariant_n}) that $\nu$ is $\mathfrak{S}(n,\infty)$-invariant measure on $E$. This measure can be extend to the $\overline{\mathfrak{S}}_\infty$-invariant measure on $X$.
\subsection{The properties of the continuous representations of the group $\overline{\mathfrak{S}}_\infty$.}
To the proof of Theorems  \ref{abelian_main_th} we will use the general facts about the continuous representations of the group $\overline{\mathfrak{S}}_\infty$, which have been well studied by A. Lieberman \cite{Lieberman} and G. Olshanski \cite{Olsh1}, \cite{Olsh2}. In this section we will give the simple constructions  of the important operators and  the  short  direct proofs of their properties.

 Let $\mathcal{K}$ be the continuous representation of $\overline{\mathfrak{S}}_\infty$ in Hilbert space $\mathcal{H}$. It follows that for each $\eta\in \mathcal{H}$
\begin{eqnarray}\label{contnuous}
\lim\limits_{k\to\infty}\sup\limits_{s\in \mathfrak{S}(k,\infty)}\|\mathcal{K}(s)\eta-\eta \|=0.
\end{eqnarray}
 Set $^n\! \sigma_m=(n+1\;\;n+m+1)(n+2\;\;n+m+2)\cdots (n+m\;\,n+2m)$, where $(k\; j)$ is a permutation that interchanges two numbers $k$, $j$ and leaves all the others fixed. We will need few auxiliary lemmas.
\begin{Lm}\label{exists limit}\label{P_def}
The sequence of the operators $\left\{ \mathcal{K}\left(\,^n\! \sigma_m \right) \right\}_{m\in \mathbb{N}}$ converges in the weak operator topology to a self-adjoint operator $P_n$.
\end{Lm}
\begin{proof}
Let us prove that the sequence $\left\{ \mathcal{K}\left(\,^n\! \sigma_m \right) \right\}_{m\in \mathbb{N}}$ is fundamental in the weak operator topology.
Assuming for the convenience that $M>m$, we write $^n\! \sigma_M$ in the form $^n\! \sigma_M=s\cdot\,^n\! \sigma_m \cdot t$, where $s,t\in \mathfrak{S}(n+m,\infty)$. Hence, using (\ref{contnuous}), we have $\lim\limits_{m,M\to\infty}\left<\left( \mathcal{K}\left( \,^n\! \sigma_M\right)-\mathcal{K}\left(^n\! \sigma_m \right)\right)\eta,\zeta \right>=0$ for all $\eta,\zeta\in \mathcal{H}$.
\end{proof}
\begin{Lm}\label{projection}
  Operator $P_n$ is a projection.
\end{Lm}
\begin{proof}
Using lemma \ref{exists limit}, for any fixed $\eta,\zeta\in \mathcal{H}$ we find the sequences  $\left\{ m_k \right\}_{k\in \mathbb{N}}$ and  $\left\{ M_k \right\}_{k\in \mathbb{N}}$ such that $m_{k+1}>m_k$, $M_k>2m_k$  and
\begin{eqnarray}\label{approximat_quadrat}
\lim\limits_{k\to\infty}\left| \left<P_n^2\eta,\zeta \right>-\left<\mathcal{K}\left(\,^n\! \sigma_{M_k} \right)\cdot\mathcal{K}\left(\,^n\! \sigma_{m_k}\right)\eta,\zeta \right> \right|=0.
\end{eqnarray}
Now we notice, that $\,^n\! \sigma_{M_k} \cdot\,^n\! \sigma_{m_k} = \,^n\! \sigma_{m_k}\cdot s_k$, where $s_k\in \mathfrak{S}\left(n+m_k,\infty \right)$. Hence, using (\ref{contnuous}) and (\ref{approximat_quadrat}), we have\newline
$0=\lim\limits_{k\to\infty}\left| \left<P_n^2\eta,\zeta \right>-\left<\mathcal{K}\left(\,^n\! \sigma_{m_k} \right)\cdot\mathcal{K}\left(s_k\right)\eta,\zeta \right> \right|
\stackrel{(\ref{contnuous})}{=}\lim\limits_{k\to\infty}\left| \left<P_n^2\eta,\zeta \right>\right.$
 \newline$\left. -\left<\mathcal{K}\left(\,^n\! \sigma_{m_k} \right)\eta,\zeta \right> \right|$ $\stackrel{\text{Lemma \ref{exists limit}}}{=}\lim\limits_{k\to\infty}\left| \left<P_n^2\eta,\zeta \right>-\left<P_n\eta,\zeta \right> \right|$.
\end{proof}
\begin{Lm}\label{n_identical}
The equality $\mathcal{K}(s)\cdot P_n=P_n$  holds for any $s\in\mathfrak{S}(n,\infty)$.
\end{Lm}
\begin{proof}
Suppose that $m> n$ and $M\geq 2m$. Then $(m\;m+1)\cdot\,^n\! \sigma_{M}=\,^n\! \sigma_{M}\cdot (m+M\;\;m+M+1)$. Hence,  applying lemma \ref{exists limit} and (\ref{contnuous}), we have\newline
$
\left< \mathcal{K}((m\;\;m+1)) P_n\eta,\zeta\right>=\lim\limits_{M\to\infty}\left< \mathcal{K}((m\;\;m+1)) \cdot\mathcal{K}(\,^n\! \sigma_{M} )\eta,\zeta\right>$\newline
$=\lim\limits_{M\to\infty}\left< \mathcal{K}(\,^n\! \sigma_{M} )\cdot\mathcal{K}((m+M\;\;m+M+1)) \eta,\zeta\right>\stackrel{\text{(\ref{contnuous})}}{=} \lim\limits_{M\to\infty}\left< \mathcal{K}(\,^n\! \sigma_{M} )\eta,\zeta\right>$ for any $\eta$, $\zeta$ $in$ $\mathcal{H}$. By lemma \ref{exists limit}, $ \mathcal{K}((m\;\;m+1))\cdot P_n=P_n$. Since the transpositions $(m\;\; m+1)$ $(m> n)$ generate the subgroup $\mathfrak{S}(n,\infty)$, lemma is proved.
\end{proof}
It follows from Lemmas \ref{P_def} and \ref{n_identical}
 that
 \begin{eqnarray}\label{projection_fixed}
 P_n\mathcal{H}=\left\{ \eta\in\mathcal{H}:\mathcal{K}(s)\eta=\eta \text{ for all } s\in\mathfrak{S}(n, \infty) \right\}.
 \end{eqnarray}
 \begin{Lm}\label{O_def}
 The sequence $\left\{ \mathcal{K}((k\;\;N)) \right\}_{N\in\mathbb{N}}$  converges in the weak operator topology to the self-adjoint projection $O_k$.
 \end{Lm}
 \begin{proof}
 Using (\ref{contnuous}) and the equality $(k\;\;N_2)=(N_1\;\;N_2)(k\;\;N_1) (k\;\;N_2)$, we obtain that the  sequence $\left\{ \mathcal{K}((k\;\;N)) \right\}_{N\in\mathbb{N}}$ is fundamental. Since  $(k\;\;N_1)(k\;\;N_2)=(k\;\;N_2)(N_1\;\;N_2)$, operator $P_k$ is a self-adjoint projection.
 \end{proof}
 \begin{Lm}
 The projections $P_n$ end $O_k$ commute: $P_nO_k=O_kP_n$.
 \end{Lm}
 \begin{proof}
 Since, by Lemma \ref{n_identical}, $O_kP_n=P_n$ for $k>n$, we suppose that $k\leq n$. By Lemmas \ref{P_def} and \ref{O_def}, for any $\eta, \zeta\in\mathcal{H}$  there exists the sequence $\left\{ M_l\right\}_{l\in \mathbb{N}}\subset\mathbb{N}$ such that $M_{k+1}>M_k$ and
 \begin{eqnarray}\label{app_P}
 \begin{split}
 \lim\limits_{l\to\infty}\left| \left< P_nO_k\eta,\zeta\right>-\left< \mathcal{K}\left(\,^n\! \sigma_{M_l}\right) O_k\eta,\zeta\right>\right|=0,\\
  \lim\limits_{l\to\infty}\left| \left< O_kP_n\eta,\zeta\right>-\left< O_k\mathcal{K}\left(\,^n\! \sigma_{M_l}\right) \eta,\zeta\right>\right|=0.
  \end{split}
 \end{eqnarray}
 For the same reason we can to find  the sequence $\left\{ N_l\right\}_{l\in \mathbb{N}}\subset\mathbb{N}$ such that $N_{k+1}>N_k>n+2M_k$ and
 \begin{eqnarray}\label{app_O}
 \begin{split}
  \lim\limits_{\to\infty}\left| \left< \mathcal{K}\left(\,^n\! \sigma_{M_l}\right) \mathcal{K}\left(k\;\;N_l \right)\eta,\zeta\right>-\left< \mathcal{K}\left(\,^n\! \sigma_{M_l}\right) O_k\eta,\zeta\right>\right|=0,\\
  \lim\limits_{l\to\infty}\left| \left<\mathcal{K}\left(k\;\;N_l \right) \mathcal{K}\left(\,^n\! \sigma_{M_l}\right)\eta,\zeta\right>-\left< O_k\mathcal{K}\left(\,^n\! \sigma_{M_l}\right) \eta,\zeta\right>\right|=0.
  \end{split}
 \end{eqnarray}
 Now, using (\ref{app_P}), (\ref{app_O}) and  the equality $\left(k\;\;N_l \right) \cdot \,^n\! \sigma_{M_l}=\,^n\! \sigma_{M_l}\cdot\left(k\;\;N_l \right) $, we obtain that
 $P_nO_k=O_kP_n$.
 \end{proof}
 \begin{Lm}\label{Lm10}
 Let $\mathfrak{S}(k,n,\infty)$ denotes the group generated by the transposition $(k\;\;n+1)$ and the subgroup $\mathfrak{S}(n,\infty)$. Then $O_kP_n$ is the self-adjoint projection on the subspace $\left\{\eta\in\mathcal{H}: \mathcal{K}(s)\eta=\eta \text{  for all } s\in \mathfrak{S}(k,n,\infty) \right\}$. In particular, $O_nP_n=P_{n-1}$ (see (\ref{projection_fixed})).
 \end{Lm}
 \begin{proof}
 The proof follows from the next chain of the equalities\newline
 $\left< \mathcal{K}((k\,\,n+1))\cdot O_kP_n\eta,\zeta \right>\stackrel{\text{Lemma \ref{O_def}}}{=}\lim\limits_{N\to\infty}\left< \mathcal{K}((k\,\,n+1)\cdot(k\;\;N))\cdot P_n\eta,\zeta \right>$\newline
 $=\lim\limits_{N\to\infty}\left< \mathcal{K}((k\;\;N))\cdot\mathcal{K}((n+1\;\;N))\cdot P_n\eta,\zeta \right>$\newline
 $\stackrel{\text{Lemma \ref{n_identical}}}{=}\lim\limits_{N\to\infty}\left< \mathcal{K}((k\;\;N))\cdot P_n\eta,\zeta \right>\stackrel{\text{Lemma \ref{O_def}}}{=}\left< O_kP_n\eta,\zeta \right>$.
 \end{proof}
Since the representation $\mathcal{K}$ is continuous, then there exists $n\in\mathbb{N}$ such that $P_n\neq 0$. Set ${\rm depth}(\mathcal{K})=\min\left\{n:P_n\neq 0  \right\}$.\label{depth}
\begin{Lm}\label{Shift_lemma}
If $n={\rm depth}(\mathcal{K})$ and $g\notin \mathfrak{S}(n,\infty)$ then $P_n\mathcal{K}(g)P_n=0$.
\end{Lm}
\begin{proof}
Let $k\leq n$ and $g(k)=m>n$. Then $g= (k\;\; m)\cdot s$, where $s(m)=m$.

 Let $\mathbb{S}=\left\{ M\in\mathbb{N}:\min\left\{ M,s^{-1}(M) \right\}>n \right\}$. It is clear that $\#\mathbb{S}=\infty$.  Under this condition we have for $M\in\mathbb{S}$ \newline
$P_n\mathcal{K}(g)P_n\stackrel{\text{Lemma \ref{n_identical}}}{=}P_n\cdot\mathcal{K}((m\;\;M))\cdot\mathcal{K}((k\;\;m))\cdot\mathcal{K}(s)\cdot\mathcal{K}((m)\;\;s^{-1}(M)))\cdot P_n$\newline
$=P_n\cdot\mathcal{K}((m\;\;M))\cdot\mathcal{K}((k\;\;m))\cdot\mathcal{K}((m\;\;M))\cdot\mathcal{K}(s)\cdot P_n =P_n\cdot\mathcal{K}((k\;\;M))\cdot\mathcal{K}(s)\cdot P_n$\newline
$\stackrel{\text{Lemma \ref{app_O}}}{=}P_n\cdot O_k\cdot\mathcal{K}(s)\cdot P_n$.

But, by (\ref{projection_fixed}) and Lemma \ref{Lm10},
\begin{eqnarray*}
\mathcal{K}((k\;\;n))\cdot P_n\cdot O_k\cdot\mathcal{K}((k\;\;n))=P_n\cdot O_n=P_{n-1}\stackrel{{\rm depth}(\mathcal{K})=n}{=}0.
\end{eqnarray*}
Therefore, $P_n\mathcal{K}(g)P_n=0$.
\end{proof}
\subsection{The Proof of Theorem \ref{abelian_main_th}}
We follow the notations of the subsection \ref{outline_abelian}.
Without loss of generality, we will to assume that $\mu$ is a probability measure.
Set $n={\rm depth}(\mathcal{K})$ (see page \pageref{depth}). Recall that we denote by $P_n$ the projection of $L^2(X,\mu)$ onto subspace $L^2_n=\left\{ \eta\in L^2(X,\mu):\mathcal{K}(s)\eta=\eta  \text{ for all } s\in\mathfrak{S}(n,\infty) \right\}$. Let operator $\mathfrak{M}(f)$, where $f\in L^\infty(X,\mu)$, acts on  $\eta\in L^2(X,\mu)$ as follows
\begin{eqnarray*}
\left( \mathfrak{M}(f)\eta \right)(x)=f(x)\eta(x).
\end{eqnarray*}
Denote by $\mathcal{N}$ von Neumann algebra generated by $\mathcal{K}(\overline{\mathfrak{S}}_\infty)$ and $\mathfrak{M}(L^\infty(X,\mu))$.  Let $\mathbb{S}$ be a subset in $L^2(X,\mu)$, and let $\left[ \mathcal{N}\mathbb{S} \right]$ be the closure of $\mathcal{N}\mathbb{S}$.

Since $\mathcal{K}$ is continuous (see subsection \ref{outline_abelian}), we have
\begin{eqnarray}\label{limit_unit}
\lim\limits_{k\to\infty}P_k=I.
\end{eqnarray}
 If $I-P_l=0$ for some $l\in\mathbb{N}\cup0$, then representation $\mathcal{K}$ is trivial; i. e. $\mathcal{K}(s)=I$ for all $s\in\overline{\mathfrak{S}}_\infty$.
For this reason, we can suppose, without loss of generality, that $P_l\neq I$ for all $l\in\mathbb{N}\cup0$.

In the sequel, we will identify the measurable subsets $\mathbb{A}$ and   $\mathbb{B}$ if their  symmetric difference $\mathbb{A}\Delta\mathbb{B}$ has zero measure.

Denote by $\widetilde{P}_k$ the orthogonal projection onto subspace $\left[ \mathcal{N}L^2_k \right]$. Since $\widetilde{P}_k$ belongs to the commutant of $\mathcal{N}$, there exists the measurable $\overline{\mathfrak{S}}_\infty$-invariant subset $X_k\subset X$ such that
\begin{eqnarray*}
\widetilde{P}_k=\mathfrak{M}(\chi_{_{X_k}}), \text{ where } \chi_{_{X_k}} \text{ is the characteristic function of } X_k.
\end{eqnarray*}
Applying (\ref{limit_unit}), we obtain
\begin{eqnarray}\label{union_X_k}
X_k\subset X_{k+1} \text{ and } \bigcup\limits_k X_k=X.
\end{eqnarray}
Consider the family of the pairwise orthogonal subspaces $H_0=L^2_n$, $H_1=\left(\widetilde{P}_{n+1}-\widetilde{P}_n\right)L^2_{n+1}$, $\ldots$, $H_j=\left(\widetilde{P}_{n+j}-\widetilde{P}_{n+j-1}\right)L^2_{n+j}$, $\ldots$.
Using the definitions of $\widetilde{P}_k$ and $L^2_k$,  we conclude from (\ref{limit_unit}) that the subspaces $\left[ \mathcal{N}H_k \right]$  are pairwise orthogonal and
\begin{eqnarray}
\bigoplus\limits_k \left[ \mathcal{N}H_k \right]=L^2(X,\mu)\text{ and } P_kH_j=0 \text{ for all } k<n+j.
\end{eqnarray}
Now we fix the orthonormal basis $\left\{ \,^i\!\eta_k  \right\}_{i=1}^{{\rm dim}\, H_k}$ in $H_k$. Denote by $\,^i\!\widetilde{P}_k$ the ortho\-go\-nal projection onto the subspace $\left[ \mathcal{N} \,^i\!\eta_k \right]\subset\left[ \mathcal{N}H_k \right]$.
Then  $\,^i\!\widetilde{P}_k=\mathfrak{M}(\chi_{_{\,^i\!X_k}})$, where  $\,^i\!X_k$ is the measurable $\overline{\mathfrak{S}}_\infty$-invariant subset in $X_k$. Since $\left\{ \,^i\!\eta_k  \right\}_{i=1}^{{\rm dim}\, H_k}$ is a basis in $H_k$, we have
\begin{eqnarray}\label{X_k_decomposition}
\bigcup\limits_{i=1}^{{\rm dim}\, H_k}\,^i\!X_k=X_{n+k}\setminus X_{n+k-1}.
\end{eqnarray}
Define the family $\left\{ \,^i Q_k  \right\}_{i=1}^{{\rm dim}\, H_k}$ of the pairwise orthogonal projections as follows
\begin{eqnarray*}
\,^1 Q_k = \,^1\!\widetilde{P}_k, \,^2 Q_k =\,^2\!\widetilde{P}_k- \,^2\!\widetilde{P}_k\cdot\,^1 Q_k,\ldots,\\
\ldots, \,^l Q_k =\,^l\!\widetilde{P}_k- \,^l\!\widetilde{P}_k\cdot\sum\limits_{i=1}^{l-1} \,^{i} Q_k, \ldots
\end{eqnarray*}
From the above it follows that
\begin{eqnarray}
\,^i\!\eta_k\in\bigoplus\limits_{j=1}^i\left[ \mathcal{N}\cdot \,^j Q_k \; \,^j\!\eta_k\right] \text{ for all } i=1,2,\ldots,{{\rm dim}\, H_k}.
\end{eqnarray}
Therefore,
\begin{eqnarray}\label{orth_sum}
\left[ \mathcal{N}H_k \right]=\bigoplus\limits_{j=1}^{{\rm dim}\, H_k}\left[ \mathcal{N}\cdot \,^j Q_k \; \,^j\!\eta_k\right].
\end{eqnarray}
The same as above, $ \,^i Q_k=\mathfrak{M}\left( \chi_{_{\,^i\!A_k}} \right)$, where $\left\{\,^i\!A_k\right\}_{i=1}^{{\rm dim}\, H_k}$ is the measurable $\overline{\mathfrak{S}}_\infty$-invariant subsets   in $X_{n+k}\setminus X_{n+k-1}$ such that $\,^i\!A_k\cap \,^j\!A_k=\emptyset$ for different $i,j$. By (\ref{X_k_decomposition}),
\begin{eqnarray}\label{difference_is_union_of_A}
\sum\limits_{i=1}^{{\rm dim}\, H_k}\,^i Q_k=\widetilde{P}_{n+k}-\widetilde{P}_{n+k-1} \text{ and } \bigcup\limits_{i=1}^{{\rm dim}\, H_k}\,^i\!A_k=X_{n+k}\setminus X_{n+k-1}.
\end{eqnarray}
Denote by $\,^i\!\mathcal{K}_k$ the restriction of the representation $\mathcal{K}$ to the subspace
\begin{eqnarray}\label{property_I_Q_K}
\,^i Q_kL^2(X,\mu)=\left[ \mathcal{N} \;^i\!\xi_k\right], \text{ where }\,^i\!\xi_k=\,^i Q_k \; \,^i\!\eta_k\;\;\; (\text{see (\ref{orth_sum})}).
\end{eqnarray}
Therefore, if $\,^i Q_k \; \,^i\!\eta_k\neq 0$ then, using the definitions of $H_k$, we obtain
\begin{eqnarray}\label{depth_i_k}
{\rm depth}\,\left( \,^i\!\mathcal{K}_k \right)=n+k.
\end{eqnarray}

Let us now build the $\overline{\mathfrak{S}}_\infty$-invariant measure $\,^i\!\nu_k$ on $\,^i\!\!A_k$.

 Since $\,^i\!\xi_k=\,^i Q_k \; \,^i\!\eta_k\in H_k$, we have
\begin{eqnarray*}
\left( \,^i\!\mathcal{K}_k(s) \,^i\!\xi_k\right)(x)=\rho(s,x)\cdot \,^i\!\xi_k(xs)=\,^i\!\xi_k(x) \text{ for each } s \in\mathfrak{S}(n+k,\infty).\;\;\;
\end{eqnarray*}
Therefore,
\begin{eqnarray}\label{modul_equality}
\rho(s,x)\cdot \left|\,^i\!\xi_k(xs)\right|=\left|\,^i\!\xi_k(x)\right| \text{ for each } s \in\mathfrak{S}(n+k,\infty).\;\;\;
\end{eqnarray}
Set $\,^i\!E_k=\left\{ x\in X:\,^i\!\xi_k(x)\neq0\right\}$. It is clear that $\,^i\!E_k\subset\,^i\!A_k$. Since \newline   $\mu\left( \left\{x\in X: \rho(g,x)=0   \right\} \right)$, we conclude from (\ref{modul_equality}) that
\begin{eqnarray}\label{fix_i_E_k}
\mu\left(\,^i\!E_k\Delta\left( \,^i\!E_k\;s \right)  \right)=0 \text{ for all } s\in\mathfrak{S}(n+k,\infty).
\end{eqnarray}

Let us prove that
\begin{eqnarray}\label{shifting}
\mu\left( (\,^i\!E_k\,g)\cap\,^i\!E_k \right)=0 \text{ for each } g\notin\mathfrak{S}(n+k,\infty).
\end{eqnarray}
 Applying (\ref{depth_i_k}) and Lemma  \ref{Shift_lemma}, we obtain
\begin{eqnarray*}
0=\left< \,^i\!\mathcal{K}_k(g)\left|\,^i\!\xi_k\right|,\left|\,^i\!\xi_k\right|\right>=\int\limits_X \rho(g,x)\left|\,^i\!\xi_k(xg)\right|\left|\,^i\!\xi_k(x)\right|\,{\rm d}\,\mu.
\end{eqnarray*}
Hence, using the equality $\mu\left( \left\{x\in X: \rho(g,x)=0   \right\} \right)=0$, we get that
\begin{eqnarray*}
\int\limits_X \left|\,^i\!\xi_k(xg)\right|\left|\,^i\!\xi_k(x)\right|\,{\rm d}\,\mu=0.
\end{eqnarray*}
Therefore, $\mu$-almost everywhere
\begin{eqnarray*}
\left|\,^i\!\xi_k(xg)\right|\left|\,^i\!\xi_k(x)\right|=0.
\end{eqnarray*}
Hence follows (\ref{shifting}).

Now we define measure $\,^i\!\mu_k$ on $X$ as follows
\begin{eqnarray}\label{i_mu_k_measure}
\,^i\!\mu_k(Y)=\mu(Y\setminus \,^i\!E_k)+\int\limits_{\,^i\!E_k} \chi_{_Y}(x)\cdot\left|\,^i\!\xi_k(x)\right|^2\,{\rm d}\,\mu.
\end{eqnarray}
Hence, assuming that $Y\subset \,^i\!E_k$ and $s\in\mathfrak{S}(n+k,\infty)$, we obtain
\begin{eqnarray}\label{inv_i_mu_k}
\begin{split}
\,^i\!\mu_k(Ys)\stackrel{(\ref{fix_i_E_k})}{=}\int\limits_{\,^i\!E_k} \chi_{_{Ys}}(x)\cdot\left|\,^i\!\xi_k(x)\right|^2\,{\rm d}\,\mu\\
=\int\limits_{\,^i\!E_k} \chi_{_{Y}}(xs^{-1})\cdot\left|\,^i\!\xi_k(x)\right|^2\,{\rm d}\,\mu\\
\stackrel{(\ref{Radon_Nikodym})}{=}\int\limits_{\,^i\!E_k}\left(\rho(s,x)\right)^2 \chi_{_{Y}}(x)\cdot\left|\,^i\!\xi_k(xs)\right|^2\,{\rm d}\,\mu\\
\stackrel{(\ref{modul_equality})}{=}\int\limits_{\,^i\!E_k} \chi_{_{Y}}(x)\cdot\left|\,^i\!\xi_k(x)\right|^2\,{\rm d}\,\mu=\,^i\!\mu_k(Y).
\end{split}
\end{eqnarray}
For the construction of the $\overline{\mathfrak{S}}_\infty$-invariant measure $\,^i\!\nu_k$ on $\,^i\!\!A_k$ we consider the right coset $H\diagdown G$, where $H=\mathfrak{S}(n+k,\infty)$ and $G=\overline{\mathfrak{S}}_\infty$. Since every bijection $s\in G$ can be write as $s=h f$, where $h\in H$ and $f\in\mathfrak{S}_\infty$ is the finite permutation, then there exists a countable full set $g_1, g_2,\ldots$ of the representatives in $G$ of the cosets $H\setminus G$. Define the map $\mathfrak{r}:H\setminus G\mapsto G$ as follows: $\mathfrak{r}(z)=g_j$, if $z=Hg_j$. We will to assume that $\mathfrak{r}(H)$ is the identity $e$ of $G$.

In the sequel, we will need the next useful equality, which follows from     (\ref{property_I_Q_K}), (\ref{fix_i_E_k}) and the definition of $\,^i\!E_k$
\begin{eqnarray}\label{i_A_k_is_union}
\,^i\!A_k=\bigcup\limits_{z\in H\diagdown G} \,^i\!E_k\;\mathfrak{r}(z).
\end{eqnarray}

For completeness, we will give below the standard algorithm of the continuation of the finite $\mathfrak{S}(n+k,\infty)$-invariant measure $\,^i\!\mu_k$ on $\,^i\!E_k$  to the $\sigma$-finite $\overline{\mathfrak{S}}_\infty$-invariant measure on $\,^i\!\!A_k$.

Take the measurable subset $Y\subset \,^i\!\!A_k$ and define its measure $\,^i\!\nu_k(Y)$ as follows
\begin{eqnarray}\label{def_i_nu_k}
\,^i\!\nu_k(Y)=\sum\limits_{z\in H\diagdown G}\,^i\!\mu_k\left(\left(Y\cap \left(\,^i\!E_k\;\mathfrak{r}(z) \right) \right)(\mathfrak{r}(z))^{-1} \right)
\end{eqnarray}
 Let us prove that
\begin{eqnarray}\label{full_inf}
\,^i\!\nu_k(Y)=\,^i\!\nu_k(Yg) \text{ for all }\;\; g\in G \text{ and } Y\subset  \,^i\!A_k.
\end{eqnarray}
For this we notice that
\begin{eqnarray*}
&\,^i\!\nu_k(Yg)=\sum\limits_{z\in H\diagdown G}\,^i\!\mu_k\left(\left((Yg)\cap \left(\,^i\!E_k\mathfrak{r}(z) \right) \right)(\mathfrak{r}(z))^{-1} \right)\\
&=\sum\limits_{z\in H\diagdown G}\,^i\!\mu_k\left(\left(Y\cap \left(\,^i\!E_k\mathfrak{r}(z) g^{-1}\right) \right)g(\mathfrak{r}(z))^{-1} \right)\\
&\stackrel{(\ref{fix_i_E_k})}{=}\sum\limits_{z\in H\diagdown G}\,^i\!\mu_k\left(\left(Y\cap \left(\,^i\!E_k\mathfrak{r}(zg^{-1}) \right) \right)g(\mathfrak{r}(z))^{-1} \right)\\
&=\sum\limits_{z\in H\diagdown G}\,^i\!\mu_k\left(\left(Y\cap \left(\,^i\!E_k\mathfrak{r}(zg^{-1}) \right) \right)\left(\mathfrak{r}(zg^{-1}) \right)^{-1}\cdot \mathfrak{r}(zg^{-1})g(\mathfrak{r}(z))^{-1} \right)\\
&=\sum\limits_{z\in H\diagdown G}\,^i\!\mu_k\left(\left(Y\cap \left(\,^i\!E_k\mathfrak{r}(z) \right) \right)\left(\mathfrak{r}(z) \right)^{-1}\cdot \mathfrak{r}(z)g(\mathfrak{r}(zg))^{-1} \right),
\end{eqnarray*}
where $\mathfrak{r}(z)g(\mathfrak{r}(zg))^{-1} \in H=\mathfrak{S}(n+k,\infty)$.
Hence, using (\ref{inv_i_mu_k}), and (\ref{def_i_nu_k}), we obtain
\begin{eqnarray*}
\,^i\!\nu_k(Yg)=\sum\limits_{z\in H\diagdown G}\,^i\!\mu_k\left(\left(Y\cap \left(\,^i\!E_k\mathfrak{r}(z) \right) \right)\left(\mathfrak{r}(z) \right)^{-1} \right)=\,^i\!\nu_k(Y).
\end{eqnarray*}
The equality (\ref{full_inf}) is proved.

Now we fix $Y\subset \,^i\!A_k$ such that $\,^i\!\nu_k(Y)=0$ and will prove that
$
\mu(Y)=0$.

Indeed, applying (\ref{def_i_nu_k}), we have
\begin{eqnarray*}
\,^i\!\mu_k\left(\left(Y\cap \left(\,^i\!E_k\;\mathfrak{r}(z) \right) \right)(\mathfrak{r}(z))^{-1} \right)=0 \text{ for all } z\in H\setminus G.
\end{eqnarray*}
It follows from (\ref{i_mu_k_measure}) that $\mu\left(\left(Y\cap \left(\,^i\!E_k\;\mathfrak{r}(z) \right) \right)(\mathfrak{r}(z))^{-1} \right)=0$ for all $z\in H\setminus G$.
Therefore, $\mu\left(\left(Y\cap \left(\,^i\!E_k\;\mathfrak{r}(z) \right) \right)\right)=0$ for all $z$. Hence, using (\ref{i_A_k_is_union}), we obtain that $\mu(Y)=0$.

Thus the restrictions of the measures $\mu$ and $\,^i\!\nu_k$ onto $\,^i\!A_k$ are equivalent. Hence, applying (\ref{difference_is_union_of_A}) and  (\ref{union_X_k}), we get that $\mu$ is equivalent to the $\overline{\mathfrak{S}}_\infty$-invariant measure $\nu=\sum\limits_{i,k}\,^i\!\nu_k$.
 Theorem \ref{abelian_main_th} is proved.
{}
B.Verkin Institute for Low Temperature Physics and Engineering\\n.nessonov@gmail.com
\end{document}